\documentclass[12pt,letterpaper]{article}
\usepackage{graphicx}
\usepackage{amsmath}
\usepackage{amsfonts}
\usepackage{amsthm}
\usepackage{amssymb}
\usepackage{color}
\def\R{\mathbb{R}}
\usepackage{verbatim} 
\usepackage{hyperref}

\newtheorem{theorem}{Theorem}[section]
\newtheorem{lemma}[theorem]{Lemma}
\newtheorem{corollary}[theorem]{Corollary}

\newcommand{\be}{\begin{equation}}
\newcommand{\ee}{\end{equation}}
\newcommand{\bea}{\begin{eqnarray}}
\newcommand{\eea}{\end{eqnarray}}
\newcommand{\beas}{\begin{eqnarray*}}
\newcommand{\eeas}{\end{eqnarray*}}

\begin{document}

\title{Connected Components of Underlying Graphs of Halving Lines}

\author{Tanya Khovanova\\MIT \and Dai Yang\\MIT}

\maketitle

\begin{abstract}
In this paper we discuss the connected components of underlying graphs of halving lines' configurations. We show how to create a configuration whose underlying graph is the union of two given underlying graphs. We also prove that every connected component of the underlying graph is itself an underlying graph.
\end{abstract}

\section{Introduction}

Halving lines have been an interesting object of study for a long time. Given $n$ points in general position on a plane the minimum number of halving lines is $n/2$. The maximum number of halving lines is unknown. The current lower bound of $O(ne^{\sqrt{\log n}})$ is found by Toth \cite{Toth}.

The current asymptotic upper bound of $O(n^{4/3})$ is proven by Dey \cite{Dey98}. In 2006 a tighter bound for the crossing number was found \cite{PRTT}, which also improved the upper bound for the number of halving lines. In our paper \cite{KY} we further tightened the Dey's bound. This was done by studying the properties of the underlying graph. 

In this paper we concentrate on the underlying graphs and properties of its connected components. 

In Section~\ref{sec:union} we use the cross construction to show how to sum two underlying graphs. In Section~\ref{sec:subtraction} we show that any connected component of the underlying graph is realizable as an underlying graphs of the halving lines of its vertices. 

\section{Definitions}

Let $n$ points be in general position in $\R^2$, where $n$ is even. A \textit{halving line} is a line through 2 of the points that splits the remaining $n-2$ points into two sets of equal size.

From our set of $n$ points, we can determine an \textit{underlying graph} of $n$ vertices, where each pair of vertices is connected by an edge if and only if there is a halving line through the corresponding 2 points.

In dealing with halving lines, we consider notions from both Euclidean geometry and graph theory. We define a \textit{geometric graph}, or \textit{geograph} for short, to be a pair of sets $(V,E)$, where $V$ is a set
of points on the coordinate plane, and $E$ consists of pairs of elements from
$V$. In essence, a geograph is a graph with each of its vertices assigned
to a distinct point on the plane.

\subsection{Examples}

\subsubsection{Four points}

Suppose we have four non-collinear points. If their convex hull is a quadrilateral, then there are two halving lines. If their convex hull is a triangle, then there three halving lines. Both cases are shown on Figure~\ref{fig:4points}.

\begin{figure}[htbp]
\begin{center}
\includegraphics[scale=0.5]{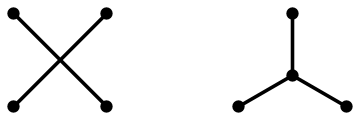}
\end{center}
  \caption{Underlying graphs for four points.}\label{fig:4points}
\end{figure}

\subsubsection{Polygon}\label{sec:polygon}

If all points belong to the convex hull of the point configuration, then each point lies on exactly one halving line. The number of halving lines is $n/2$, and the underlying graph is a matching graph --- a union of $n/2$ disjoint edges. The left side of Figure~\ref{fig:4points} shows an example of this configuration. 

For any point configuration there is at least one halving line passing through each vertex. Hence, the polygon provides an example of the minimum number of halving lines, and an example of the most number of disconnected components.

\section{Union of Connected Components}\label{sec:union}

Given two underlying graphs of two halving lines configurations, the following construction allows to create a new halving line configuration whose underlying graph consists of two given graphs as connected components.

\subsection{Cross}

The following construction we call a \textit{cross}. Given two sets of points, with $n_1$ and $n_2$ points respectively whose underlying graphs are $G_1$ and $G_2$, the cross is the construction of $n_1+n_2$ points on the plane whose underlying graph has two isolated components $G_1$ and $G_2$. 

We squeeze the initial sets of points in $G_1$ and $G_2$ into long narrow segments, a process called segmentarizing (see \cite{KY}). Note that segmentarizing is an affine transform, and does not change which pairs of points form halving lines. Then we intersect these segments in such a way that the halving lines of $G_1$ split the vertices of $G_2$ into two equal halves, and vice versa (See Figure~\ref{fig:cross}).

\begin{figure}[htbp]
\begin{center}
\includegraphics[scale=0.5]{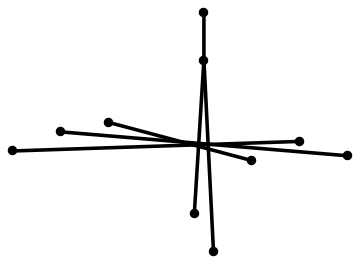}
\end{center}
  \caption{The Cross construction.}\label{fig:cross}
\end{figure}

With respect to geographs, the image of the cross construction depends on the precise manner in which $G_1$ and $G_2$ are segmentarized and juxtaposed. However, with respect to underlying graphs, the cross construction defines an associative and commutative binary operation.

Our Polygon example in subsection~\ref{sec:polygon} can be viewed as the cross construction of a 2-path graph with itself many times.

It is interesting to note that in the cross construction the halving lines of one component divide the points of the other component into the same halves. It is not necessarily so. Two connected components can interact in a way different from a cross as seen in Figure~\ref{fig:noncross}.

\begin{figure}[htbp]\label{fig:noncross}
\begin{center}
\includegraphics[scale=0.3]{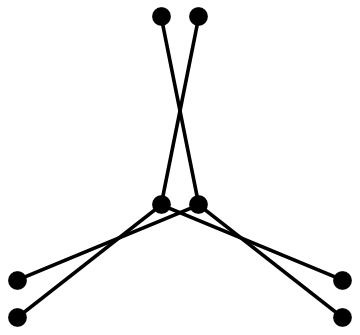}
\end{center}
  \caption{Two connected components that are not formed through the cross construction.}\label{fig:NonCrossSum}
\end{figure}

\section{Decomposition of Connected Components}\label{sec:subtraction}

We will now prove that graph composition has an inverse of sorts, namely that we can subtract disconnected components of an underlying graph. We will show that a connected component of the underlying graph is itself an underlying graph.

But before doing so, we will introduce some definitions.

Given a set of points $G$ and a directed line, we can orient $G$ and pick a direction to be North; and thus, we can define the East and the West half of the plane. We define the \textit{$G$-balance} of the line to be the difference between the number of West points and East points in $G$. Similarly, we can define the \textit{$G$-balance} of two points as the $G$-balance of the line through them. It is often does not matter which direction is chosen as North, but it is important that when we move a variable line, the two sides of the line move accordingly.

Let $A$ be a union of connected subcomponents in $G$. We will prove that the halving lines of $G$ that are formed by points in $A$ are also halving lines in $A$.

\begin{theorem}\label{thm:preserved}
If $A$ is a union of some connected subcomponents in $G$, then for every pair of points in $A$ forming a halving line in $G$, their $A$-balance is zero.
\end{theorem}

\begin{proof}
Let $A$ contain $k$ of the halving lines of $G$. Label these lines $l_1,l_2,...,l_k$ by order of counter-clockwise orientation. For any two such lines $l_i,l_{i+1}$, there is a unique rotation of at most $180$ degrees about their point of intersection that maps $l_i$ to $l_{i+1}$, where the indices are taken mod $k$. Define $R_i$ to be the open region swept by $l_i$ as it moves into $l_{i+1}$ under this rotation. Note that each $R_i$ consists of two symmetric unbounded sectors.

We claim that there are no vertices of $A$ lying in any of the $R_i$. Assume that $R_i$ contains a vertex $P$ of $A$. Draw the lines through $P$ parallel to $l_i$ and $l_{i+1}$, and call them $m_i$ and $m_{i+1}$ respectively, see Figure~\ref{fig:balance}. Note that neither $m_i$ nor $m_{i+1}$ are halving lines of $G$. Take a variable line $m$ through $P$ to initially coincide with $m_i$, and rotate it counter-clockwise until it coincides with $m_{i+1}$. The side of $m_i$ that contains $l_i$ has more points of $G$ than the other side. Similarly, the side of $m_{i+1}$ that contains $l_{i+1}$ has more points than the other side. As $m$ rotates its $G$-balance will change sign, so by continuity, $m$ coincides with a halving line during this rotation, a halving line which should occur between $l_i$ and $l_{i+1}$ in the counter-clockwise ordering of halving lines. As one point on the line, namely $P$, belongs to $A$, the other point of this halving line also belongs to $A$. Hence, $R_i$ can not contain a vertex of $A$. 

\begin{figure}[htbp]
\begin{center}
\includegraphics[scale=0.4]{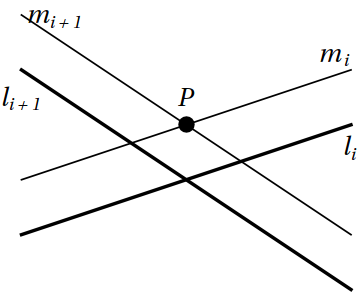}
\end{center}
  \caption{No points of $A$ exist in the regions $R_i$.}\label{fig:balance}
\end{figure}

As $l_i$ rotates into $l_{i+1}$, it does not sweep across any points of $A$ along the way, and furthermore the points of $A$ on $l_i$ or $l_{i+1}$ do not affect the net $A$-balance of these lines. When the line $l_1$ completes its $180^\circ$ rotation, its $A$-balance does not change due to the above argument, but by definition it should be negated, so it must be zero.
\end{proof}

\begin{corollary}\label{thm:componentmix}
If an underlying geograph $G$ is composed of disconnected components $A$ and $B$, then every halving line in $A$ divides points in $B$ in half, and vice versa.
\end{corollary}

Given a geograph and a fixed orientation such that no edges are vertical, we can denote the \textit{left-degree} and \textit{right-degree} of a given vertex as the number of edges emanating from the left and the right of that vertex respectively. Then the following result follows from the existence of structures called chains found on any oriented underlying geograph \cite{Dey98}, \cite{KY}. We will not discuss the definition of chains here. We will only mention that each chain is a subpath in the underlying graph that travels from left-to-right. We note that chains have the following properties:

\begin{itemize}
\item A vertex on the left half of the underlying graph is a left endpoint of a chain.
\item A vertex on the right half of the underlying graph is a right endpoint of a chain.
\item Every vertex is the endpoint of exactly one chain.
\item Every halving line is part of exactly one chain.
\end{itemize}

Now we are ready to prove the following lemma.

\begin{lemma}\label{thm:leftright}
Let $G$ be an underlying geograph with a fixed orientation. If $v$ is a vertex appearing among the left half (right half) of $G$, then the right-degree (left-degree) of $v$ is one more than the left-degree (right-degree) of $v$.
\end{lemma}
\begin{proof}
If $v$ appears among the left half of the vertices of $G$, then every chain passing through $v$ contributes one left-degree and one right-degree to $v$. There is one chain with $v$ as an endpoint, and it must emanate on the right since chains cannot end among the $\frac{n}{2}$ leftmost vertices of $G$. Therefore, the right-degree of $v$ exceeds the left-degree of $v$ by one. The proof when $v$ appears among the $\frac{n}{2}$ rightmost vertices of $G$ is analogous.
\end{proof}

The previous theorems and lemmas allow us to prove our main result of this section that the subtraction works:

\begin{theorem}\label{thm:subtraction}
Suppose that an underlying geograph $G$ contains a union of connected components $A$. Then if all vertices of $G$ that do not belong to $A$ are removed, the halving lines of $A$ in $G$ are precisely the halving lines of $A$ by itself.
\end{theorem}

\begin{proof}
Fix an orientation of $G$, and consider $A$ by itself under the same orientation. Since Lemma~\ref{thm:preserved} asserts that deleting the extra vertices preserves the existing halving lines of $A$, it suffices to show that no new halving lines are added. Assume the contrary, and call $E_A$ the set of new edges in $A$ which were not in $G$. Let $v$ and $w$ be the leftmost and rightmost vertices of $A$ with edges in $E_A$, respectively. Clearly $v$ lies to the left of $w$ in $A$, and hence in $G$ as well. Note that $v$ has a greater right-degree in $A$ than in $G$, but the same left-degree in both geographs. Therefore, by Lemma ~\ref{thm:leftright}, $v$ was not among the leftmost half of the vertices of $G$, so it must have been among the rightmost half. Similarly, $w$ has a greater left-degree in $A$ than in $G$, but the same right-degree in both geographs. Thus, $w$ must have been among the leftmost half of the vertices in $G$. But this contradicts the fact that $v$ must lie to the left of $w$ in $G$, so our statement holds.
\end{proof}

\begin{corollary}
Each connected component of an underlying geograph $G$ is itself an underlying geograph.
\end{corollary}

\section{Properties of Connected Components}

Connected components of the underlying graph are themselves underlying graphs. Hence, the properties of the underlying graphs are shared by each component. For example, every connected component has at least three leaves. 

In addition, the properties of chains with respect to any geograph are the same as the properties of these chains with respect to the connected component they belong to.

Consequently, we present a stronger version of Corollary~\ref{thm:componentmix}. 

\begin{lemma}
In any orientation of a geograph $G$, if $C$ is a connected component of $G$, then the left half of the vertices of $C$ belong to the left half of the vertices of $G$, and the right half of the vertices of $C$ belong to the right half of the vertices of $G$.
\end{lemma}

\begin{proof}
A vertex $v\in C$ is on the left half of $C$ or $G$ iff its right-degree is one more than its left-degree in $C$ or $G$. But the left-degree and right-degree of $v$ is the same whether we consider the entirety of $G$ or only its connected component. Hence, $v$ is on the left half of $C$ iff it is on the left half of $G$.\end{proof}

\section{Acknowledgements}

We are thankful to the UROP at MIT that provided financial support for the second author and to Professor Jacob Fox for supervising the project.

\end{document}